\documentclass[12pt]{amsart}

\usepackage{cite}
\usepackage{tikz}
\usetikzlibrary{knots}
\usepackage[colorlinks, linkcolor=blue, citecolor=blue, hyperindex,breaklinks]{hyperref}
\hypersetup{nesting=true,debug=true,naturalnames=true}
\usepackage[margin=.9in]{geometry}

\usetikzlibrary{decorations.markings}

\newcommand{\Z}{{\mathbb{Z}}}

\newtheorem{theorem}{Theorem}[section]
\newtheorem{lemma}[theorem]{Lemma}

\theoremstyle{definition}
\newtheorem{definition}[theorem]{Definition}
\newtheorem{proposition}[theorem]{Proposition}
\newtheorem{corollary}[theorem]{Corollary}
\newtheorem{example}[theorem]{Example}

\theoremstyle{remark}
\newtheorem{remark}[theorem]{Remark}

\numberwithin{equation}{section}

\begin{document}

\title{Involutions of alternating links} 

\author{Keegan Boyle}  

\email{kboyle@math.ubc.ca}


\commby{David Futer}

\setcounter{section}{0}

\begin{abstract}
Let $L$ be an alternating prime non-split link in $S^3$. We study compositions of flypes between reduced alternating diagrams for $L$ to classify involutions on $L$. As consequences, we show that when $L$ is 2-periodic its quotient is alternating and that if $L$ is freely 2-periodic then it has an even number of components.
\end{abstract}
\maketitle
\section{Introduction}
A diagram for a link in $S^3$ is \emph{alternating} if the crossings alternate over and under following the strands of the diagram (see Figure \ref{fig:rotational} for an example), and a link $L$ is \emph{alternating} if it has such a diagram. The purpose of this paper is to give a diagrammatic classification of involutions on alternating links using flypes. For us, an involution $\tau$ is a smooth orientation-preserving $\Z/2$ action on $S^3$ which preserves $L$ set-wise. By a classical result of Smith \cite{smith}, $\tau$ either has an empty fixed-point set, in which case it is conjugate to the antipodal action, or a fixed set $S^1$. If $\tau$ has fixed set $S^1$, then the Smith conjecture, proved in \cite{MB}, implies that $\tau$ is conjugate to rotation around an unknotted axis. Thus involutions on links are generally classified into the following three categories.
\begin{enumerate}
\item When $\tau$ is conjugate to a rotation around an axis disjoint from $L$, we say that $L$ is 2\emph{-periodic}. In this case, the quotient of $L$ is another link in $S^3$. 
\item When $\tau$ is conjugate to a rotation around an axis which intersects $L$, the quotient is no longer a link - the intersection points become the endpoints of arcs in the quotient. If $L$ is a knot we say that $L$ is $\emph{strongly invertible}$.
\item Finally, when $\tau$ is conjugate to the antipodal action on $S^3$, we say that $L$ is \emph{freely} 2\emph{-periodic}. In this case the quotient is a link in $\mathbb{R}P^3$.
\end{enumerate}
We will use \emph{periodic} without qualification to refer only to the first case, and will specify \emph{freely periodic} to refer to the third case. 

The main result of this paper is Theorem \ref{thm:main}, which will allow us to construct all involutions on alternating links by beginning with one of three basic types of involutions on a diagram (see Section \ref{sec:basic}) and performing a nested or disjoint sequence of the replacements shown in Figures \ref{fig:transform1} and \ref{fig:transform2}. If the link is periodic, such a diagram then allows us to construct an alternating diagram for the quotient.

This theorem was foreshadowed by Menasco and Thistlethwaite's work in \cite{MT}; in \cite{HTW} they state: ``any symmetry of a prime alternating link must be visible, up to flypes, in any alternating diagram of the link". However, there is still a meaningful reduction in showing exactly which sequences of flypes are possible in a finite order symmetry. This is demonstrated, for example, by a conjecture in \cite{JN} that if an alternating knot $K$ has crossing number $c$ and odd prime period $q$, then $c|q$. This conjecture was proved only recently by Costa and Hongler \cite{CH} and independently by the author \cite{B}. The corresponding statement for 2-periods is false (consider the trefoil which is 2-periodic but has crossing number 3), but we are nonetheless able to prove the related result that the quotient of an alternating non-split 2-periodic link is alternating (see Corollary \ref{cor:quotient}). The underlying reason for the additional complication in the case of involutions is the existence of a diagrammatic involution for a period which reverses the orientation of the projection sphere (see Figure \ref{fig:reflective}).

Combining our results with \cite{B} or \cite{CH} we then have the following theorem (see Definition \ref{def:symmetry} for a precise definition of a period on $L$). 

\begin{theorem}
\label{thm:quotient}
The quotient of an alternating non-split periodic link $L$ is alternating.
\end{theorem}

For an application of this theorem consider \cite[Theorem 6]{B2}, which gives a relationship between the Alexander polynomial of a 2-periodic alternating knot and its quotient using a spectral sequence on knot Floer homology developed by Hendricks \cite{H} and refined by Hendricks, Lipshitz, and Sarkar \cite{HLS}. Theorem \ref{thm:quotient} strengthens \cite[Theorem 6]{B2} by removing the assumption that the quotient knot is alternating:

\begin{theorem} (a slight improvement of \cite[Theorem 6]{B2})
Let $\widetilde{K}$ be an alternating 2-periodic knot with quotient $K$ which has linking number $\lambda$ with the axis of symmetry. Consider the symmetrized Alexander polynomials
\[
\Delta_{\widetilde{K}}(t) = \widetilde{a}_0 + \sum_{\widetilde{s} > 0} \widetilde{a}_{\widetilde{s}} (t^{\widetilde{s}} + t^{-\widetilde{s}}), \mbox{ and } \Delta_{K(t)} = a_0 + \sum_{s > 0} a_s (t^s + t^{-s}).
\]
Then for each $s$,
\[
|\widetilde{a}_{2s + \frac{\lambda -1}{2}} - \widetilde{a}_{2s + \frac{\lambda +1}{2}}| \geq a_s.
\]
\end{theorem}

As an additional application of our classification, we will see that a freely 2-periodic prime non-split alternating link must have an even number of components (see Corollary \ref{cor:free}). Specifically, no alternating prime knots are freely 2-periodic. It is interesting to compare this to a result of Sakuma \cite[Main Theorem]{S} that amphicheiral hyperbolic knots cannot be freely periodic, and a recent construction by Paoluzzi and Sakuma \cite{PS} of prime amphicheiral (non-hyperbolic) freely 2-periodic knots. 

The main technical tool used in this paper is Menasco and Thistlethwaite's theorem that any self-homeomorphism of pairs $(S^3,L) \to (S^3,L)$ is isotopic to a composition of flypes \cite[Main Theorem]{MT}, see Theorem \ref{thm:mt}.

\subsection{Organization}
Section \ref{sec:basic} defines some elementary involutions on alternating diagrams. Section \ref{sec:flype} gives some background on flypes, including some results in the particular situation of involutions. Section \ref{sec:main} gives a precise statement and proof of the classification as well as some interesting corollaries.

\subsection{Acknowledgments}
I would like to thank Liam Watson, Robert Lipshitz, and Ahmad Issa for helpful comments and conversations. I would also like to thank the referee for many comments which improved the clarity of the paper. 

\section{Basic Diagrammatic Involutions \label{sec:basic}}
In this section, we define three basic involutions which are visible in a link diagram. These basic involutions are the building blocks for the main classification. Throughout this paper, all links will be prime, alternating, and non-split. To begin, we specify our exact meaning of a symmetry of $L$. 

\begin{definition}
\label{def:symmetry}
Given a link $L$, consider the group of self-diffeomorphisms of pairs $(S^3,L)$ which are required to be orientation-preserving on $S^3$, but not on $L$. The \emph{symmetry group} of $L$ is the quotient of this group by the (normal) subgroup of diffeomorphisms which are isotopic to the identity. 

A \emph{symmetry} of a link $L$ is a finite order element of the symmetry group. That is, it is a finite order self-diffeomorphism of $S^3$ fixing $L$ set-wise and not isotopic to the identity. An \emph{involution} of $L$ is an order 2 symmetry. A \emph{period} of $L$ is a symmetry with fixed set an unknot disjoint from $L$. A \emph{strong inversion} of $L$ is an involution with fixed set an unknot intersecting each component of $L$ in two points. A \emph{free period} of $L$ is a symmetry which acts freely on $S^3$.
\end{definition}

We note that this definition disagrees slightly with the more general notion of symmetry which allows diffeomorphisms which are isotopic to the identity. Specifically, the torus links $T(p,q)$ have a periodic symmetry of order $n$ if $n|p$ or $n|q$, but this symmetry is trivial in the symmetry group defined above. Our main conclusions, Theorem \ref{thm:quotient} and Corollary \ref{cor:free}, hold in this case as well, however. Indeed, the only alternating torus links are $T(2,n)$ for which the symmetries are well understood. On the other hand for non-torus knots $K$, \cite[Theorem 10.6.6]{Kaw} implies that no finite order $(>1)$ self-diffeomorphism of $S^3$ fixing $K$ set-wise can be isotopic to the identity; in this case these notions of symmetry agree.

\begin{definition}
Given a link $L \subset S^3$ or a tangle $T \subset B^3$ and a finite cyclic symmetry $\tau$ of $L$ or $T$, an \emph{intravergent} diagram for $(L,\tau)$ or $(T,\tau)$ is a diagram such that $\tau$ is conjugate to rotation within the plane of the diagram. See Figure \ref{fig:rotational} for an example. 
\end{definition}
\begin{figure}
\includegraphics{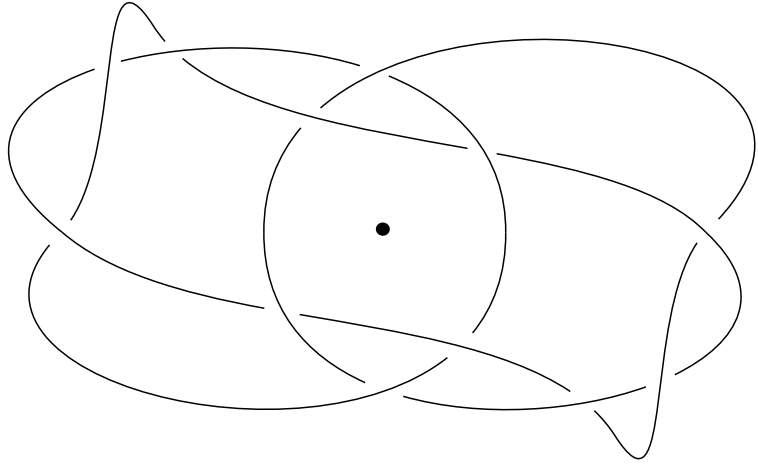}
\caption{An intravergent diagram for a periodic involution on an alternating knot. The axis of rotation is perpendicular to the page, and is shown as a dot in the center.}
\label{fig:rotational}
\end{figure}
Note that any period (with no fixed points on the link) or any involution with a pair of fixed points (with the axis of symmetry intersecting the plane at a crossing) can be shown with an intravergent diagram, but an intravergent diagram cannot show an involution with fixed points on multiple components of the link. If $\tau$ is a period, this is commonly called a \emph{periodic} diagram.
\begin{definition}
Given a link $L \subset S^3$ or a tangle $T \subset B^3$ and an involution $\tau$ on $L$ or $T$, a \emph{transvergent} diagram for $(L,\tau)$ or $(T, \tau)$ is a diagram such that $\tau$ is conjugate to a rotation around an axis contained within the plane of the diagram. See Figure \ref{fig:reflective} for an example.
\end{definition}
\begin{figure}
\includegraphics{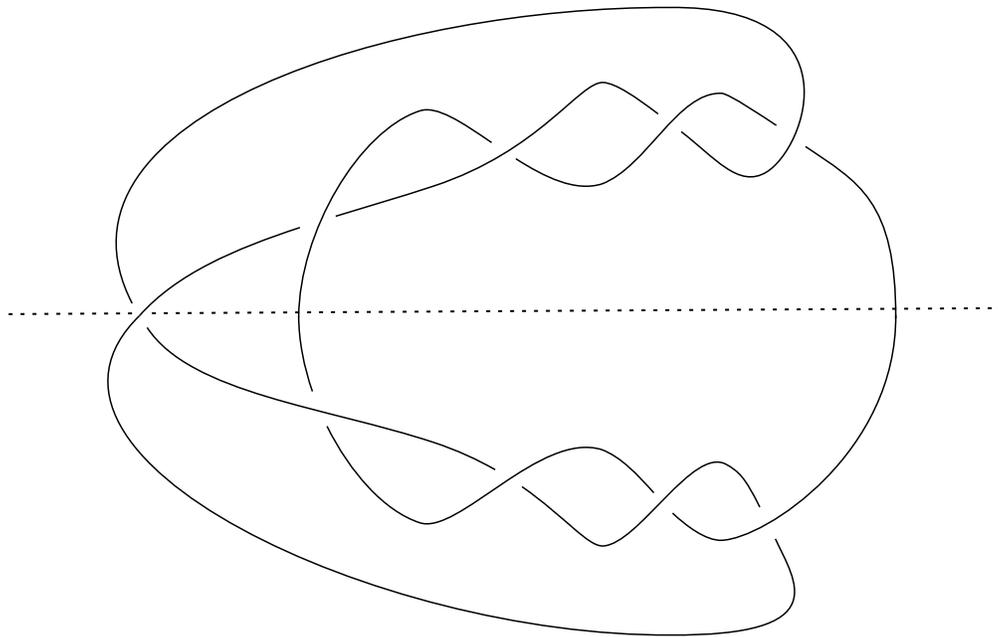}
\caption{A transvergent diagram for a strong inversion on an alternating knot. The axis of rotation is the dashed line.}
\label{fig:reflective}
\end{figure}
Note that an involution with any number of fixed points on the link can be shown in a transvergent diagram (unlike an intravergent diagram). For this reason, this is the more general type of symmetric diagram to use for involutions of links. See for example \cite{LW} where a transvergent diagram is called involutive.

\begin{remark}
The terms \emph{intravergent} and \emph{transvergent} were constructed from the latin root \emph{vergo} (I turn) and the prefixes \emph{intra} (within) and \emph{trans} (through). Thus \emph{intravergent} refers to a rotation within the plane of the diagram and \emph{transvergent} refers to a rotation through the plane of the diagram.
\end{remark}
\begin{definition}
Given an alternating link $L \subset S^3$ and a free involution on $S^3$ which fixes $L$, a \emph{freely periodic} alternating diagram for $(L, \tau)$ is a diagram for $L$ consisting of two identical tangles with two additional crossings connected in the configuration shown in Figure \ref{fig:free}. 
\end{definition}
\begin{remark}
A more general freely periodic diagram includes a half twist on $n$ strands (see for example \cite{C} or \cite[Figure 1(b)]{PY}). Since a half twist on $n$ strands will be alternating only if $n = 1,2$ and prime only if $n \neq 1$, we only consider the case of 2 strands in this paper.
\end{remark}
\begin{figure}
\scalebox{1.3}{\includegraphics{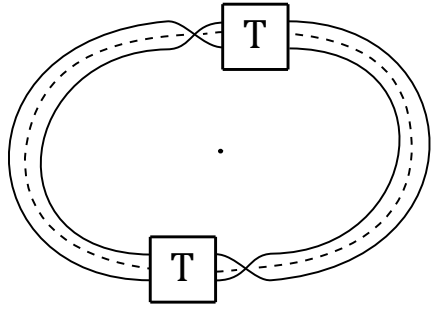}}
\caption{A freely 2-periodic alternating diagram. Here $T$ can be any alternating tangle and the shown double-points should be resolved into crossings so that the entire diagram is alternating. The free involution is given by a $\pi$ rotation around the dashed axis, and then a $\pi$ rotation within the plane of the diagram around the shown dot. Note that this diagram necessarily has an even number of components.}
\label{fig:free}
\end{figure}

\section{Flypes \label{sec:flype}}
In this section we discuss compositions of flypes on a prime alternating link $L \subset S^3$. This section is largely similar to \cite[Section 2]{B}, but we reproduce it here for clarity. We begin with some definitions.
\begin{definition}
The \emph{standard crossing ball} $B_{\mbox{\tiny{std}}} = (B^3, D^2, a_1 \cup a_2)$ is the triple of 
\begin{enumerate}
\item the 3-ball $\{(x,y,z) \in \mathbb{R}^3 \mid |(x,y,z)| \leq 1\}$,
\item the horizontal unit disk inside this ball $\{(x,y,z) \in \mathbb{R}^3 \mid z=0 \mbox{ and } |(x,y,z)| \leq 1\}$, and
\item the union of the two arcs $a_1 = \{(x,y,z) \in \mathbb{R}^3 \mid x=0, z \geq 0 \mbox{ and } y^2 + z^2 = 1\}$ and $a_2 = \{(x,y,z) \in \mathbb{R}^3 \mid y=0, z\leq 0 \mbox{ and } x^2 + z^2 = 1\}$.
\end{enumerate}
\end{definition}
\begin{definition}
A \emph{realized diagram} $\lambda(D)$ for a link $L \subset S^3$ is a collection of smooth embeddings
\begin{enumerate}
\item $S^2 \hookrightarrow S^3$, the \emph{projection sphere},
\item $L: (S^1)^n \hookrightarrow S^3$, the \emph{link}, and
\item $\{B_i\} \hookrightarrow S^3$, the \emph{crossing balls},
\end{enumerate}
such that the $\{B_i\}$ are disjoint and $L \subset S^2 \cup \{B_i\}$, along with homeomorphisms of triples $c_i: (B_i, B_i \cap S^2, B_i \cap L)\to B_{\mbox{\tiny{std}}}$, the \emph{crossing ball identification maps}. The \emph{diagram} $D$ is the labeled graph in $S^2$ which is the projection of $L$ with vertices labeled to indicate under and over crossings.
\end{definition}
We will also apply the term \emph{realized diagram} to a tangle by restricting the above definition to a disk $D^3 \subset S^3$ which $L$ intersects transversely and for which $\partial D^3$ is disjoint from all crossing balls. 

\begin{definition}
An \emph{isomorphism} of realized diagrams $g: \lambda(D) \to \lambda(\overline{D})$ is a homeomorphism of pairs $g: (S^3, L) \to (S^3, \overline{L})$ such that $g(S^2)$ is isotopic to $\overline{S}^2$ relative to $\overline{L}$, $g(B_i) = \overline{B}_i$, and either
\begin{enumerate}
\item $g$ preserves the orientation of $S^2$, and $\overline{c}_i \circ g= c_i$, or
\item $g$ reverses the orientation of $S^2$, and $\overline{c}_i \circ g = m \circ c_i$,
\end{enumerate}
where $m: B_{\mbox{\tiny{std}}} \to B_{\mbox{\tiny{std}}}$ is $\pi$ rotation around the line $x=y$ and $z = 0$. In case (1) we refer to the isomorphism as \emph{orientation preserving}, and in case (2) we refer to the isomorphism as \emph{orientation reversing}. All isomorphisms must preserve the orientation on $S^3$, but need not preserve the orientation on $L$.
\end{definition}
It is immediate that $\lambda(D)$ and $\lambda(E)$ are isomorphic realized diagrams if and only if $D$ and $E$ are isomorphic diagrams. We also note that the three basic involutions from the previous section all induce automorphisms of realized diagrams. For an intravergent involution, the automorphism is orientation preserving; for transvergent and freely periodic involutions, the automorphism is orientation reversing. In fact, these are the only involutions which can be constructed as an automorphism of realized diagrams.

\begin{lemma}
\label{lemma:floop}
An involution of a realized diagram $g:\lambda(D) \to \lambda(D)$ is either intravergent, transvergent, or freely periodic. Furthermore, if $D$ is a tangle diagram (as opposed to a link diagram) then $g$ is not freely periodic.
\end{lemma}
\begin{proof}
There are only three conjugacy classes of involutions on $S^2$, which are classified by their fixed sets. They are
\begin{enumerate}
	\item the antipodal map, which arises from the antipodal map on $S^3$ and corresponds to a freely periodic diagram,
	\item reflection across a circle, which arises from rotation around that circle in $S^3$ and corresponds to a transvergent diagram, and
	\item rotation around a pair of points, which arises from rotation around an axis through those points in $S^3$ and corresponds to an intravergent diagram.
\end{enumerate}
Since case (1) has no set-wise fixed disk, it cannot give rise to a tangle diagram.
\end{proof}

\begin{definition}
A \emph{standard flype} is a transformation between realized diagrams $\lambda(D) \to \lambda(E)$ of the form shown in Figure \ref{fig:flype}. More precisely, it is a homeomorphism $s:S^3 \to S^3$ which restricts to the identity on a round ball containing $T_2$ (shown as the exterior of $\alpha_2$), $\pi$ rotation around the horizontal axis on a ball containing $T_1$ (shown as the interior of $\alpha_1$) and a linear homotopy between them. We will refer to the crossing ball $B_1$ (or the respective crossing in the underlying diagram) as the crossing \emph{removed} or \emph{destroyed} by the flype, and to the crossing ball $B_2$ as the crossing \emph{created} by the flype.
\end{definition}

\begin{definition}
\label{definition:flype}
A \emph{flype} $f: \lambda(D) \to \lambda(D')$ is any composition $f = g_2 \circ s \circ g_1$ where $g_1$ and $g_2$ are isomorphisms of realized diagrams, and $s$ is a standard flype. We will refer to the crossing ball created by $f$, $g_2(B_2)$ in $\lambda(D')$, as $c(f)$, and the crossing ball destroyed by $f$, $g_1^{-1}(B_1)$ in $\lambda(D)$, as $d(f)$. The ball $\alpha_2$ containing the tangle $T_1$ and the crossing ball $B_1$ will be referred to as the \emph{domain} of $f$. 
\end{definition}
Our notions of flype and standard flype are slightly more general than the original definitions in \cite{MT}, making Theorem \ref{thm:mt} nominally weaker. This is sufficient for our purposes, however.

We now introduce some notation for compositions of flypes. 
\begin{definition}
A \emph{composable collection of flypes} is an ordered set $F = (f_n,f_{n-1},\dots,f_1)$ of flypes which can be composed in the same order. We refer to the automorphism of pairs $(S^3,L)$ given by their composition $f_n \circ f_{n-1} \circ \dots \circ f_1$ as $F_{\circ}$. Two composable collections of flypes $F$ and $H$ are $\emph{equivalent}$ if $F_{\circ} = H_{\circ}$. We will refer to a set containing a single isomorphism of realized diagrams as a composable collection of zero flypes.
\end{definition}

We also need the following equivalence relation.
 
\begin{definition}
Let $F = (f_n, f_{n-1}, \dots, f_1)$ with $f_i: \lambda(D_i) \to \lambda(D_{i+1})$, and consider the set of all crossing balls in all $D_i$. We generate an equivalence relation on this set by declaring that $[B_i] = [f_i(B_i)]$ for all $B_i$ in $D_i$. If two crossing balls $B_i$ and $B_j$ are in the same equivalence class under this relation we say that $B_i$ and $B_j$ are \emph{equivalent mod} $F$, and write $[B_i]_F = [B_j]_F$. If $F$ is clear from context, then we simply write $[B_i] = [B_j]$. 
\end{definition}
This equivalence can be thought of as identifying crossings $B_i$ and $B_j$ in different diagrams if a sequence of flypes takes $B_i$ to $B_j$.

We will now give some basic facts about flypes. Consider a flype $f: \lambda(D) \to \lambda(D')$. In the underlying planar graph $D$ we get a distinguished vertex corresponding to the crossing ball $d(f)$ and a distinguished pair of edges $(e^1_{c(f)}, e^2_{c(f)})$ corresponding to the crossing ball $c(f)$. That is, the flype replaces this pair of edges with a vertex corresponding to $c(f)$ and four edges connecting this vertex to the vertices at the ends of $e^1_{c(f)}$ and $e^2_{c(f)}$. Similarly, we have an edge pair $(e^1_{d(f)},e^2_{d(f)})$ in $D'$ corresponding to $d(f)$. 
\begin{definition}
Let $D \subset S^2$ be a 4-valent graph on $S^2$ with vertices labeled to indicate crossing types so that $D$ is a standard link diagram. Let $\gamma$ in $S^2$ be an oriented closed curve which intersects $D$ in two edges $(e^1_{c(f)},e^2_{c(f)})$ and a vertex $d(f)$. A \emph{diagrammatic flype} $f: D \to D'$ replaces $(e^1_{c(f)},e^2_{c(f)})$ with $c(f)$ and four incident edges as described above and replaces $d(f)$ and its incident edges with $(e^1_{d(f)},e^2_{d(f)})$. The flype $f$ also reflects the so-far-unchanged part of $D$ in the interior of $\alpha$ across a line connecting $c(f)$ and $d(f)$ (see Figure \ref{fig:flype}).
\end{definition}

The following lemma states that a diagrammatic flype is enough to reconstruct a flype. 
\begin{lemma}
\label{lemma:unique}
Let $f: \lambda(D) \to \lambda(D')$ and $g: \lambda'(D) \to \lambda'(D')$ be flypes such that $(e^1_{c(f)}, e^2_{c(f)}) = (e^1_{c(g)},e^2_{c(g)})$ and $d(f) = d(g)$, or inversely such that $(e^1_{d(f)}, e^2_{d(f)}) = (e^1_{d(g)},e^2_{d(g)})$ and $c(f) = c(g)$. Then there exists a pair of isomorphisms of realized diagrams $g_1,g_2$ such that $f = g_1 \circ g \circ g_2$.
\end{lemma}
\begin{proof}
To begin, note that there is an isomorphism between $\lambda(D)$ and $\lambda'(D)$, so that we may consider $f$ and $g$ to start at the same realized diagram. Similarly, there is an isomorphism from $\lambda(D')$ to $\lambda'(D')$, so we may assume $f$ and $g$ end at the same realized diagram. Note that $f$ and $g$ induce the same diagrammatic flype on $D$. In particular $f$ and $g$ restrict to the same map on crossing balls, since both are determined by the crossing ball identification maps for $\lambda(D)$ and $\lambda(D')$. From there we have a unique extension to the rest of $S^3$ (up to an isomorphism of realized diagrams) by Alexander's Theorem.
\end{proof}

\begin{figure}
\scalebox{0.7}{
\begin{tikzpicture}
\begin{knot}[
clip width = 15, 
flip crossing = 1,
]
\strand[black , thick] (8,0) .. controls +(1,0) and +(0,1) .. (9,-1) .. controls +(0,-2) and +(-3,-3) .. (0,0) .. controls +(0,0) and +(-1,0) .. (2,1);
\strand[black , thick] (8,1) .. controls +(1,0) and +(0,-1) .. (9,2) .. controls +(0,2) and +(-3,3) .. (0,1) .. controls +(0,0) and +(-1,0) .. (2,0);
\end{knot}
\draw[black, ultra thick] (2, -.53) -- (2,1.5) -- (4,1.5) -- (4,-.5) -- (2,-.5);
\draw[black, thick] (4,1) -- (6,1);
\draw[black, thick] (4,0) -- (6,0);
\draw[black, ultra thick] (6, -.53) -- (6,1.5) -- (8,1.5) -- (8,-.5) -- (6,-.5);

\draw[black, thick, dotted] (1,.5) .. controls +(0,2.5) and +(0,2.5) .. (4.5,.5);
\draw[black, thick, dotted] (4.5,.5) .. controls +(0,-2.5) and +(0,-2.5) .. (1,.5);

\draw[black, thick, dotted] (0,.5) .. controls +(0,3.5) and +(0,3.5) .. (5.5,.5);
\draw[black, thick, dotted] (5.5,.5) .. controls +(0,-3.5) and +(0,-3.5) .. (0,.5);

\node at (3,.5) {\huge{$T_1$}};
\node at (7,.5) {\huge{$T_2$}};

\node at (4.2, 2){$\alpha_1$};
\node at (4.3, -2){$\alpha_2$};
\node at (0.65,0.85){$B_1$};
\node at (4,-4){\huge{$\lambda(D)$}};
\end{tikzpicture}}
\scalebox{0.7}{
\begin{tikzpicture}
\begin{knot}[
clip width = 15, 
]
\strand[black, thick] (4,1) .. controls +(1,0) and +(-1,0) .. (6,0);
\strand[black, thick] (4,0) .. controls +(1,0) and +(-1,0) .. (6,1);
\end{knot}
\draw[black, ultra thick] (2, -.53) -- (2,1.5) -- (4,1.5) -- (4,-.5) -- (2,-.5);
\draw[black, thick, dotted] (1,.5) .. controls +(0,2.5) and +(0,2.5) .. (4.5,.5);
\draw[black, thick, dotted] (4.5,.5) .. controls +(0,-2.5) and +(0,-2.5) .. (1,.5);

\draw[black, thick, dotted] (0,.5) .. controls +(0,3.5) and +(0,3.5) .. (5.5,.5);
\draw[black, thick, dotted] (5.5,.5) .. controls +(0,-3.5) and +(0,-3.5) .. (0,.5);
\draw[black, ultra thick] (6, -.53) -- (6,1.5) -- (8,1.5) -- (8,-.5) -- (6,-.5);

\draw[black , thick] (8,0) .. controls +(1,0) and +(0,1) .. (9,-1) .. controls +(0,-2) and +(-2,-2) .. (0,-1) .. controls +(1,1) and +(-1,0) .. (2,0);
\draw[black , thick] (8,1) .. controls +(1,0) and +(0,-1) .. (9,2) .. controls +(0,2) and +(-2,2) .. (0,2) .. controls +(1,-1) and +(-1,0) .. (2,1);

\node at (3,.5) {\raisebox{\depth}{\scalebox{1}[-1]{\huge{$T_1$}}}};
\node at (7,.5) {\huge{$T_2$}};

\node at (4.2, 2){$\alpha_1$};
\node at (4.3, -2){$\alpha_2$};
\node at (5,0.95){$B_2$};
\node at (4,-4){\huge{$\lambda(E)$}};
\end{tikzpicture}}
\caption{A standard flype fixes the exterior of $\alpha_2$ and rotates the interior of $\alpha_1$ by $\pi$ around the horizontal axis, with a linear homotopy in between. It removes the crossing ball $B_1$ and creates the crossing ball $B_2$.}
\label{fig:flype}
\end{figure}
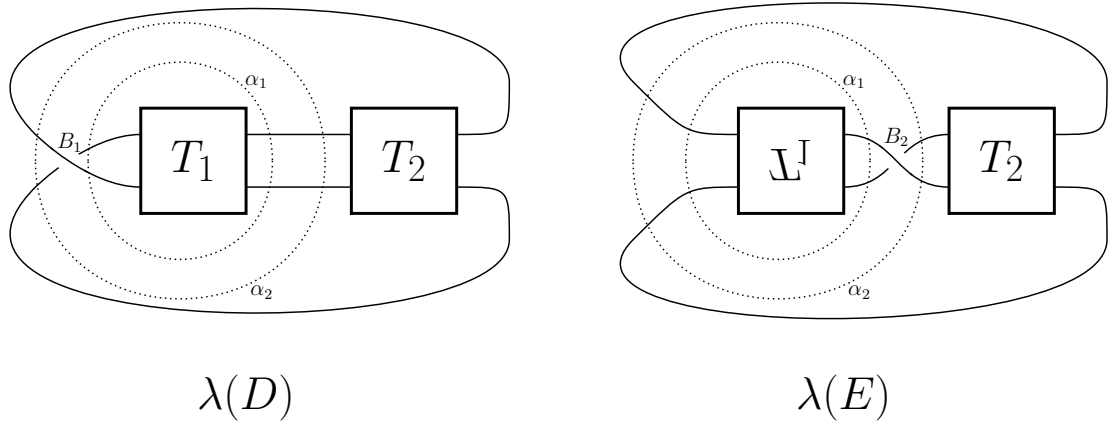

\begin{lemma}
\label{lemma:exists}
Let $\lambda(D)$ and $\lambda(D')$ be realized reduced alternating diagrams for $L$, and let $N(L)$ be a regular neighborhood of $L$. If $f: \lambda(D) \to \lambda(D')$ is a homeomorphism $S^3 \to S^3$ such that $f$ agrees with a flype when restricted to $(N(L)\cap S^2)\cup\{c_i\}$, then $f$ is a flype. That is, if $f$ induces a diagrammatic flype on the underlying diagram, then $f$ is a flype.
\end{lemma}
\begin{proof}
Let $\varphi:\lambda(D) \to \lambda(D')$ be a flype, so that $f$ and $\varphi$ agree when restricted to both the crossing balls and a neighborhood of $L$ in $S^2$. Let $g = f \circ \varphi^{-1}:\lambda(D') \to \lambda(D')$, and observe that $g$ is the identity map on each crossing ball and on $N(L) \cap S^2$. In particular, this determines the relative isotopy class of $g(S^2)$ so that $g$ is an isomorphism of realized diagrams. But then $f = g \circ \varphi$, so $f$ is a flype as desired.
\end{proof}
By combining Lemmas \ref{lemma:unique} and \ref{lemma:exists} we see that a diagrammatic description of a flype determines it uniquely up to isomorphism of realized diagrams. We now turn to minimizing the complexity of a collection of composable flypes in the following sense.

\begin{definition}
Let $F$ be a composable collection of flypes $(f_n, f_{n-1},\dots, f_1)$, and let $x_i$ be the number of crossing balls in the domain of $f_i$. Then the \emph{complexity} of $F$ is defined to be
\[
\mbox{cx}(F) = n + \sum_{i=1}^n x_i.
\] 
The collection $F$ is \emph{reduced} if cx($F$) is minimal among collections equivalent to $F$.
\end{definition}
Note that the complexity of a collection $F$ is 0 if and only if the collection contains no flypes so that $F_{\circ}$ is an isomorphism of realized diagrams. In particular, isomorphisms of realized diagrams are reduced.

\begin{figure}
\scalebox{0.7}{
\begin{tikzpicture}
\begin{knot}[
clip width = 15, 
flip crossing = 1,
]
\strand[black , thick] (11,0) .. controls +(1,0) and +(0,1) .. (12,-1) .. controls +(0,-2) and +(-3,-3) .. (0,0) .. controls +(0,0) and +(-1,0) .. (2,1);
\strand[black , thick] (11,1) .. controls +(1,0) and +(0,-1) .. (12,2) .. controls +(0,2) and +(-3,3) .. (0,1) .. controls +(0,0) and +(-1,0) .. (2,0);
\end{knot}
\draw[black, ultra thick] (2, -.53) -- (2,1.5) -- (4,1.5) -- (4,-.5) -- (2,-.5);
\draw[black, ultra thick] (5, .5) -- (5,1.5) -- (6,1.5) -- (6,.5) -- (5,.5);
\draw[black, ultra thick] (7, .5) -- (7,1.5) -- (8,1.5) -- (8,.5) -- (7,.5);
\draw[black, thick] (4,1) -- (5,1);
\draw[black, thick] (6,1) -- (7,1);
\draw[black, thick] (8,1) -- (9,1);
\draw[black, thick] (4,0) -- (9,0);
\draw[black, ultra thick] (9, -.53) -- (9,1.5) -- (11,1.5) -- (11,-.5) -- (9,-.5);

\draw[black, thick, dotted] (4.5,1) .. controls +(0,1.2) and +(0,1.2) .. (8.5,1);
\draw[black, thick, dotted] (8.5,1) .. controls +(0,-1.2) and +(0,-1.2) .. (4.5,1);

\draw[black, thick, dotted] (.5,.5) .. controls +(0,3) and +(0,3) .. (6.5,.5);
\draw[black, thick, dotted] (6.5,.5) .. controls +(0,-3) and +(0,-3) .. (.5,.5);

\node at (3,.5) {\huge{$T_2$}};
\node at (10,.5) {\huge{$T_1$}};
\node at (7.5,1){$T_3$};
\node at (5.5,1){$T_4$};

\node at (3, 2.3){$\alpha_{f_1}$};
\node at (7.5, 2){$\alpha_{f_2}$};
\node at (6,-4){\huge{$(A)$}};
\end{tikzpicture}}
\scalebox{0.7}{
\begin{tikzpicture}
\begin{knot}[
clip width = 15, 
flip crossing = 1,
]
\strand[black , thick] (8,0) .. controls +(1,0) and +(0,1) .. (9,-1) .. controls +(0,-2) and +(-3,-3) .. (0,0) .. controls +(0,0) and +(-1,0) .. (2,1);
\strand[black , thick] (8,1) .. controls +(1,0) and +(0,-1) .. (9,2) .. controls +(0,2) and +(-3,3) .. (0,1) .. controls +(0,0) and +(-1,0) .. (2,0);
\end{knot}
\draw[black, ultra thick] (2, -.53) -- (2,1.5) -- (4,1.5) -- (4,-.5) -- (2,-.5);
\draw[black, thick] (4,1) -- (6,1);
\draw[black, thick] (4,0) -- (6,0);
\draw[black, ultra thick] (6, -.53) -- (6,1.5) -- (8,1.5) -- (8,-.5) -- (6,-.5);

\draw[black, thick, dotted] (5,.5) .. controls +(0,2.5) and +(0,2.5) .. (8.5,.5);
\draw[black, thick, dotted] (8.5,.5) .. controls +(0,-2.5) and +(0,-2.5) .. (5,.5);

\draw[black, thick, dotted] (.5,.5) .. controls +(0,3) and +(0,3) .. (5,.5);
\draw[black, thick, dotted] (5,.5) .. controls +(0,-3) and +(0,-3) .. (.5,.5);

\node at (3,.5) {\huge{$T_2$}};
\node at (7,.5) {\huge{$T_1$}};

\node at (4, 2){$\alpha_{f_1}$};
\node at (6.3, -1.7){$\alpha_{f_2}$};
\node at (4,-4){\huge{$(B)$}};
\end{tikzpicture}} \hskip -.5 in.
\scalebox{0.7}{
\begin{tikzpicture}
\begin{knot}[
clip width = 15, 
flip crossing = 1,
]
\strand[black , thick] (8,0) .. controls +(1,0) and +(0,1) .. (9,-1) .. controls +(0,-2) and +(-3,-3) .. (0,0) .. controls +(0,0) and +(-1,0) .. (2,1);
\strand[black , thick] (8,1) .. controls +(1,0) and +(0,-1) .. (9,2) .. controls +(0,2) and +(-3,3) .. (0,1) .. controls +(0,0) and +(-1,0) .. (2,0);
\end{knot}
\draw[black, ultra thick] (2, -.53) -- (2,1.5) -- (4,1.5) -- (4,-.5) -- (2,-.5);
\draw[black, thick] (4,1) -- (6,1);
\draw[black, thick] (4,0) -- (6,0);
\draw[black, ultra thick] (6, -.53) -- (6,1.5) -- (8,1.5) -- (8,-.5) -- (6,-.5);

\draw[black, thick, dotted] (5,.5) .. controls +(0,2.5) and +(0,2.5) .. (8.5,.5);
\draw[black, thick, dotted] (8.5,.5) .. controls +(0,-2.5) and +(0,-2.5) .. (5,.5);

\draw[black, thick, dotted] (.5,.5) .. controls +(0,3.5) and +(0,3.5) .. (8.5,.5);
\draw[black, thick, dotted] (8.5,.5) .. controls +(0,-3.5) and +(0,-3.5) .. (.5,.5);

\node at (3,.5) {\huge{$T_2$}};
\node at (7,.5) {\huge{$T_1$}};

\node at (3, 2.5){$\alpha_{f_1}$};
\node at (6, -1){$\alpha_{f_2}$};
\node at (4,-4){\huge{$(C)$}};
\end{tikzpicture}}

\caption{Three potential configurations for the composition of two flypes $f_1:\lambda(D) \to \lambda(D')$ and $f_2:\lambda(D') \to \lambda(D'')$ with $c(f_1) = d(f_2)$. The diagrams shown are $D$ (as opposed to $D'$ or $D''$).}
\label{fig:orientations}
\end{figure}
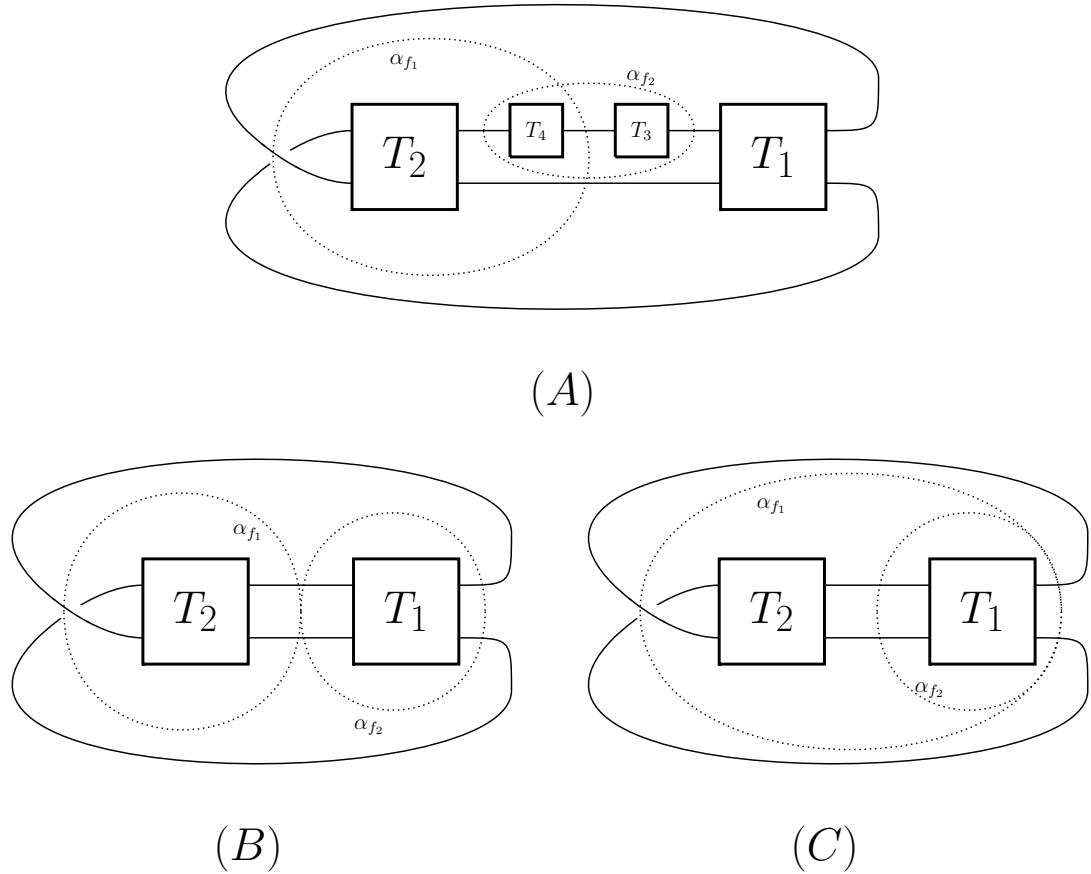

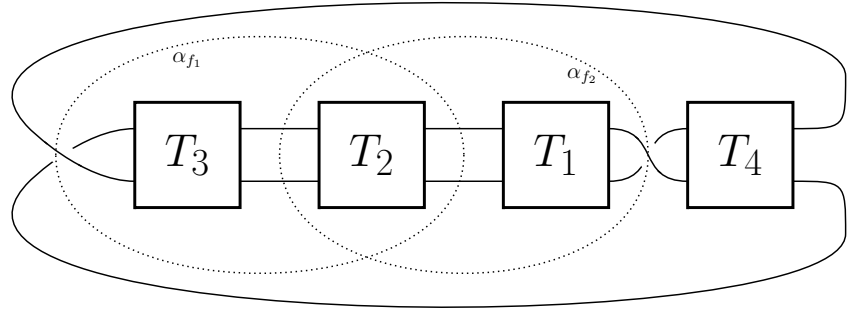
\begin{figure}
\scalebox{0.7}{
\begin{tikzpicture}
\begin{knot}[
clip width = 15, 
flip crossing = 1,
]
\strand[black , thick] (14.5,0) .. controls +(1,0) and +(0,1) .. (15.5,-1) .. controls +(0,-2) and +(-3,-3) .. (0,0) .. controls +(0,0) and +(-1,0) .. (2,1);
\strand[black , thick] (14.5,1) .. controls +(1,0) and +(0,-1) .. (15.5,2) .. controls +(0,2) and +(-3,3) .. (0,1) .. controls +(0,0) and +(-1,0) .. (2,0);
\strand[black , thick] (11,1) .. controls +(1,0) and +(-1,-0) .. (12.5,0);
\strand[black , thick] (11,0) .. controls +(1,0) and +(-1,0) .. (12.5,1);
\end{knot}
\draw[black, ultra thick] (2, -.53) -- (2,1.5) -- (4,1.5) -- (4,-.5) -- (2,-.5);
\draw[black, ultra thick] (5.5, -.53) -- (5.5,1.5) -- (7.5,1.5) -- (7.5,-.5) -- (5.5,-.5);
\draw[black, ultra thick] (12.5, -.53) -- (12.5,1.5) -- (14.5,1.5) -- (14.5,-.5) -- (12.5,-.5);
\draw[black, thick] (4,1) -- (5.5,1);
\draw[black, thick] (7.5,1) -- (9,1);
\draw[black, thick] (4,0) -- (5.5,0);
\draw[black, thick] (7.5,0) -- (9,0);
\draw[black, ultra thick] (9, -.53) -- (9,1.5) -- (11,1.5) -- (11,-.5) -- (9,-.5);

\draw[black, thick, dotted] (.5,.5) .. controls +(0,3) and +(0,3) .. (8.25,.5);
\draw[black, thick, dotted] (8.25,.5) .. controls +(0,-3) and +(0,-3) .. (.5,.5);

\draw[black, thick, dotted] (4.75,.5) .. controls +(0,3) and +(0,3) .. (11.75,.5);
\draw[black, thick, dotted] (11.75,.5) .. controls +(0,-3) and +(0,-3) .. (4.75,.5);

\node at (3,.5) {\huge{$T_3$}};
\node at (10,.5) {\huge{$T_1$}};
\node at (6.5,.5) {\huge{$T_2$}};
\node at (13.5,.5) {\huge{$T_4$}};

\node at (3, 2.3){$\alpha_{f_1}$};
\node at (10.5, 2){$\alpha_{f_2}$};
\end{tikzpicture}}
\caption{A possible configuration for two flypes $f_1: \lambda(D) \to \lambda(D')$ and $f_2:\lambda(D') \to \lambda(D)$ whose domains overlap. The shown diagram is $D$, $\alpha_{f_1}$ is the domain for $f_1$, and $\alpha_{f_2}$ is the domain for $f_2$.}
\label{fig:commute}
\end{figure}
\begin{definition}
Let $F = (f_2, f_1)$. Then the domains $\alpha_1$ and $\alpha_2$ of $f_1$ and $f_2$ are \emph{essentially disjoint} if $\alpha_1 \cap \alpha_2$ does not contain any crossing balls. Similarly, $\alpha_1$ and $\alpha_2$ are \emph{essentially nested} if for each crossing ball $B_1 \subset \alpha_1$ there is a crossing ball $B_2 \subset \alpha_2$ with $[B_1] = [B_2]$ or vice versa.
\end{definition}

\begin{proposition}
\label{prop:permute}
Let $F = (f_n, f_{n-1}, \dots, f_1)$ be a reduced composable collection of flypes. Then for all $i,j$, the domains of $f_i$ and $f_j$ are essentially disjoint or essentially nested. Furthermore, for any permutation $\phi$ of $\{1,2,\dots n\}$, there are flypes $f_{\phi(n)}', f_{\phi(n-1)}', \dots, f_{\phi(1)}'$ such that $F' = (f_{\phi(n)}', f_{\phi(n-1)}', \dots, f_{\phi(1)}')$ is equivalent to $F$, $F'$ is reduced, and for each $i$, $[c(f_i)] = [c(f_i')]$ and $[d(f_i)] = [d(f_i')]$.
\end{proposition}
In order to prove this proposition, we first consider some specific cases.

\begin{lemma}
\label{lemma:combine}
If $f_1: \lambda(D_1) \to \lambda(D_2)$ and $f_2: \lambda(D_2) \to \lambda(D_3)$ are flypes such that $c(f_1) = d(f_2)$ and $(e^1_{c(f_2)},e^2_{c(f_2)}) \neq (e^1_{d(f_1)},e^2_{d(f_1)})$, then there exists a flype $f_{1,2}: \lambda(D_1) \to \lambda(D_3)$ with $f_2 \circ f_1 \cong f_{1,2}$ and cx$((f_{1,2})) <$ cx$((f_2, f_1))$. In particular, $(f_2, f_1)$ is not reduced.
\end{lemma}
\begin{proof}
By Lemmas \ref{lemma:unique} and \ref{lemma:exists}, it is enough to consider diagrammatic flypes. There are three possible configurations for $f_2$ relative to $f_1$ (see Figure \ref{fig:orientations}). Observe however that configuration (A) is impossible, since if $T_3$ or $T_4$ is a non-trivial tangle then $L$ cannot be prime. 

In configuration (B) the composition will flip both tangles $T_1$ and $T_2$ over once, so that $f_{1,2}$ can be defined by having domain a ball containing $T_1$, $T_2$, and $d(f_1)$. Similarly, in configuration (C) the composition will flip the tangle $T_1$ over twice so that $f_{1,2}$ can be defined by having domain a ball containing $T_2$ and $d(f_1)$. In both cases it is clear that cx($(f_{1,2})) <$ cx$((f_2, f_1))$. 
\end{proof}

\begin{lemma}
\label{lemma:commute}
If $f_1: \lambda(D_1) \to \lambda(D_2)$ and $f_2: \lambda(D_2) \to \lambda(D_3)$ are flypes such that $c(f_1) \neq d(f_2)$ and $(f_2, f_1)$ is reduced, then there exists a pair of flypes $f_2': \lambda(D_1) \to \lambda(D_2')$ and $f_1': \lambda(D_2') \to \lambda(D_3)$ such that 
\begin{enumerate}
\item $f_2 \circ f_1 \cong f_1' \circ f_2'$, and
\item for $i \in\{1,2\}$, $[d(f_i)] = [d(f_i')]$ and $[c(f_i)] = [c(f_i')]$.
\end{enumerate}
\end{lemma}

\begin{proof}
Again, by Lemmas \ref{lemma:unique} and \ref{lemma:exists} it is enough to prove this lemma diagrammatically on the underlying graphs. On one hand, consider the case where either the domain for $f_1$ is essentially nested in the domain for $f_2$ or else the domains for $f_1$ and $f_2$ are essentially disjoint. Define $f_1'$ and $f_2'$ to have the same domains as $f_1$ and $f_2$ respectively. Then it is clear that $d(f_1) = (f_2')^{-1}(d(f_1'))$ and $c(f_2) = f_1'(c(f_2'))$ so that $[d(f_1)] = [d(f_1')]$ and $[c(f_2)] = [c(f_2')]$, and similarly for $c(f_1)$ and $d(f_2)$.

On the other hand, suppose that the domains are neither essentially disjoint nor essentially nested. Then we have the configuration shown in Figure \ref{fig:commute}. In this case, define $f_1'$ as the flype with domain containing $T_1$ and define $f_2'$ as the flype with domain containing $T_3$. We see that $f_2 \circ f_1 \cong f_1' \circ f_2'$ but cx$((f_1', f_2')) < $ cx($(f_2, f_1))$, so $(f_2, f_1)$ is not reduced.
\end{proof}

\begin{proof}[Proof of Proposition \ref{prop:permute}]
We begin by proving the claim about permutations. Since permutations are generated by transpositions between adjacent elements, it is enough to replace $f_i \circ f_{i-1}$ with $f_{i-1}'\circ f_i'$. Note that since $F$ is reduced, so is $(f_i, f_{i-1})$.

On one hand, suppose that $c(f_{i-1}) \neq d(f_i)$. Then we can apply Lemma \ref{lemma:commute} directly, and we are done. On the other hand, suppose that $c(f_{i-1}) = d(f_i)$. Then we have two cases. If $(e^1_{c(f_i)},e^2_{c(f_i)}) \neq (e^1_{d(f_{i-1})},e^2_{d(f_{i-1})})$, we can apply Lemma \ref{lemma:combine} to see that $(f_i, f_{i-1})$ is not reduced. If $(e^1_{c(f_i)},e^2_{c(f_i)}) = (e^1_{d(f_{i-1})},e^2_{d(f_{i-1})})$, then we have two possible configurations for $f_i$ and $f_{i-1}$. Either the domains of $f_i$ and $f_{i-1}$ are essentially disjoint apart from the crossing $c(f_{i-1})$, or else they are essentially equal (in that they are each essentially nested in the other). In both cases, $f_i \circ f_{i-1}$ is equivalent to an isomorphism of realized diagrams, and hence $(f_i, f_{i-1})$ is not reduced.

We have also seen that the only reduced configurations for $(f_i, f_{i-1})$ (and also $(f_{i-1}', f_i')$) have essentially nested or essentially disjoint domains. By applying the above permutation argument, we can then see that this must be true for any pair $f_i$ and $f_j$.
\end{proof}

\begin{definition}
Consider a reduced collection of composable flypes $F = (f_n, f_{n-1},\dots, f_1)$ and let $\alpha_i$ be the domain of $f_i$. A flype $f_i$ is \emph{innermost in} $F$ (or simply \emph{innermost}) if for each $j \neq i$, the domain $\alpha_i$ is either essentially disjoint from $\alpha_j$ or essentially nested in $\alpha_j$.
\end{definition}

\begin{lemma}
\label{lemma:innermost}
In any reduced collection of composable flypes $F$ with cx$(F) > 0$, there is an innermost flype.
\end{lemma}
\begin{proof}
First observe that since cx$(F) > 0$, $F$ contains at least one flype. Then by Proposition \ref{prop:permute} any pair of flypes in $F$ have essentially disjoint or essentially nested domains. This forms a partial order by declaring $f_i < f_j$ if the domain of $f_i$ is essentially nested in the domain of $f_j$. Any minimal element under this partial order is innermost.
\end{proof}

\section{A Classification of Involutions}\label{sec:main}
In this section, we will use flypes to prove the classification of involutions from alternating diagrams via the basic involutions from Section \ref{sec:basic}. Before stating the main theorem, we give a definition.
\begin{definition}
Consider an involution $\tau$ on a diagram $D$ equal to the composition $g\circ f$ of a flype $f$ and another map $g$. We say that $f$ is \emph{type A} if there exists a 4-ended tangle $T' \subset D$ containing the domain of $f$ such that $\tau$ restricted to $T'$ is given by $h \circ f$, where $h$ is an intravergent $\pi$ rotation of $T'$. See the right side of Figure \ref{fig:transform1} where $f$ is the flype on $T$. Similarly we say that $f$ is \emph{type B} if the same is true, but $h$ is a transvergent rotation. See the right side of Figure \ref{fig:transform2}.
\end{definition}
Consider an intravergent 4-ended tangle diagram with a central 4-ended subtangle consisting of two arcs with no crossings (see the left side of Figure \ref{fig:transform1}). Observe that by replacing this central subtangle with a crossing and a transvergent tangle (see the right side of Figure \ref{fig:transform1}), we get a new tangle with an involution given by the composition of a flype with an intravergent rotation. A similar replacement can be done for a transvergent tangle by replacing a trivial central tangle with an intravergent tangle and a crossing (see Figure \ref{fig:transform2}).
\begin{theorem}\label{thm:main}
Let $L$ be a prime alternating non-split link in $S^3$ and $\tau$ be an involution on $S^3$ preserving $L$. Then there is a reduced alternating diagram for $(L, \tau)$ which is either intravergent, transvergent, freely periodic, or constructed from an intravergent or a transvergent diagram by a performing a (finite) sequence of replacements as shown in Figure \ref{fig:transform1} and Figure \ref{fig:transform2}.
\begin{figure}
\scalebox{1.3}{\includegraphics{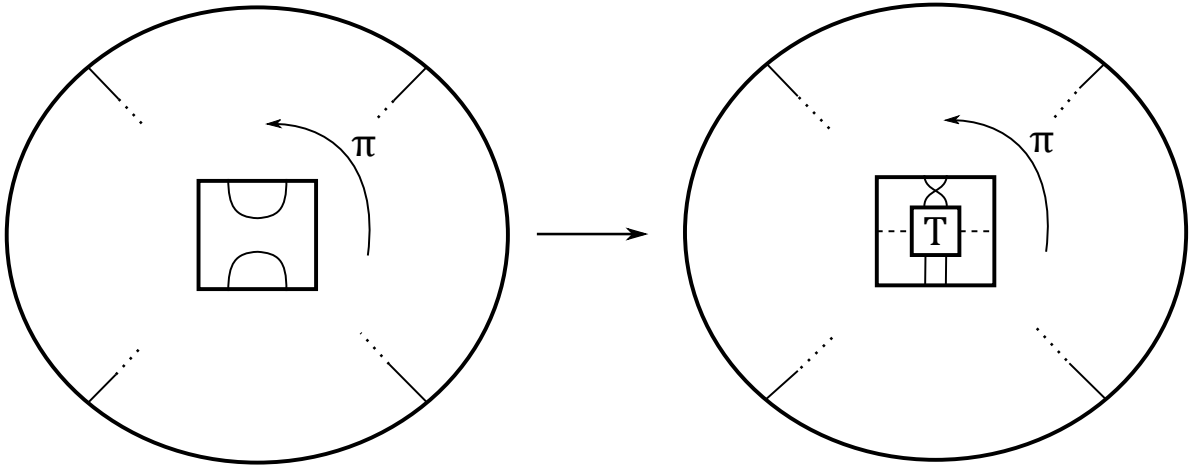}}
\caption{A transformation to create an alternating 4-ended tangle with a new involution from an intravergent alternating 4-ended tangle. The tangle on the right is constructed by replacing the inner crossingless tangle as shown. Here $T$ must be alternating and transvergent around the dashed axis, and the intersecting arcs shown should be a crossing of whichever type makes the diagram alternating. The involution on the right is given by performing a flype on $T$, and then rotating the entire diagram by $\pi$ within the plane of the diagram.}
\label{fig:transform1}
\end{figure}

\begin{figure}
\scalebox{1.3}{\includegraphics{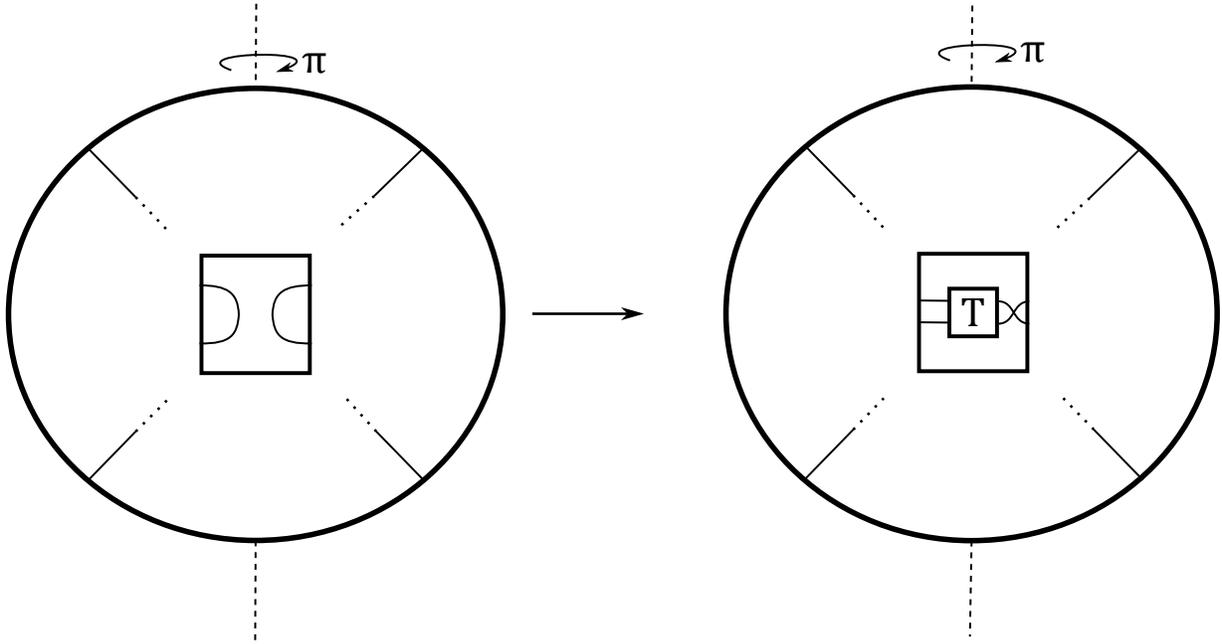}}
\caption{Another transformation from one alternating 4-ended tangle with an involution to another. On the left is given a transvergent alternating 4-ended tangle with a crossingless 4-ended central subtangle. The diagram on the right is constructed by replacing the inner crossingless tangle as shown. Here $T$ must be alternating and intravergent, and the intersecting arcs shown should be a crossing of whichever type makes the diagram alternating. The involution on the right is then given by performing a flype on $T$, and then rotating the entire diagram by $\pi$ around the dashed axis.}
\label{fig:transform2}
\end{figure}
\end{theorem}

Before proving this theorem we will give an illustrative example and prove some interesting corollaries.

\begin{example}
Consider Figure \ref{fig:example}. Notice that taking the tangle in the exterior of the $\beta$ box and connecting the loose ends with a trivial tangle (either way) gives an intravergent diagram, and that the interior of the $\alpha$ box is a transvergent diagram, so that the entire diagram is an example of the transformation from Figure \ref{fig:transform1}. In particular, an involution on this knot is given by performing a flype on the $\alpha$ box, then rotating the entire diagram by $\pi$. This diagram as a whole however, is neither transvergent nor intravergent. 
\begin{figure}
\scalebox{.75}{\includegraphics{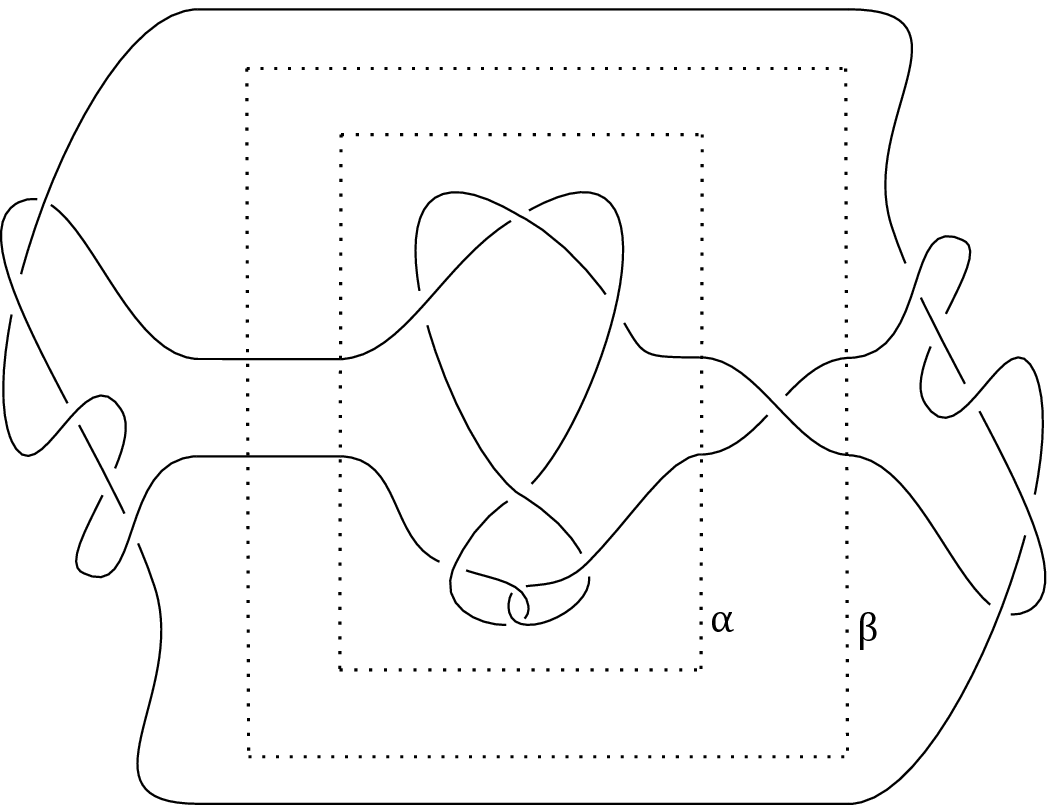}}
\scalebox{.75}{\includegraphics{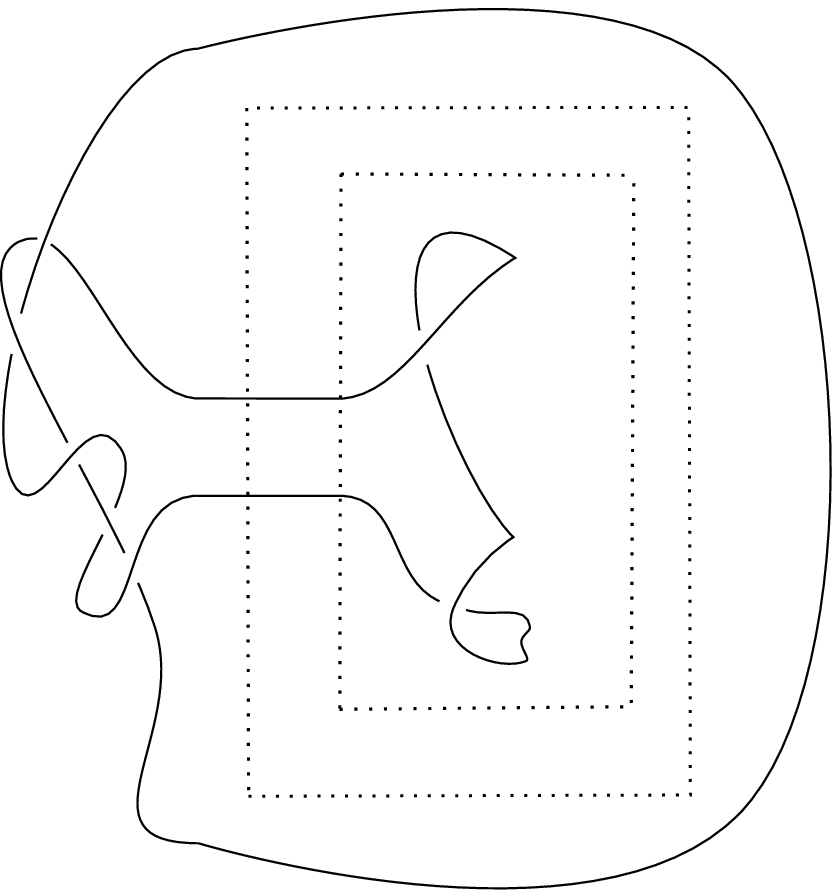}}
\caption{A reduced alternating diagram of a knot with a non-basic periodic involution (left) and its quotient (right).}
\label{fig:example}
\end{figure}
\end{example}

\begin{corollary}\label{cor:quotient}
The quotient of an alternating 2-periodic prime non-split link is alternating. 
\end{corollary}
\begin{proof}
We induct on the number of type A and type B replacements needed to construct the diagram guaranteed by Theorem \ref{thm:main}. For the base case we have an intravergent or a transvergent diagram which clearly has an alternating quotient diagram. Now suppose that we have a diagram $D$ as given in Theorem \ref{thm:main} and we perform a replacement of type A (see Figure \ref{fig:transform1}) or type B (see Figure \ref{fig:transform2}) to obtain a new diagram $D'$. The new transvergent or intravergent tangle $T$ then has an alternating quotient diagram, and by inductive assumption the diagram $D$ has an alternating quotient diagram. Furthermore, the quotient of $D'$ is the boundary sum of the quotients of $D$ and $T$ in the way that preserves the alternating condition. The extra crossing from the flype (which is not part of either $D$ or $T$) will become nugatory in the quotient if you choose a fundamental domain which contains it, or it is possible to choose a fundamental domain avoiding the extra crossing as we have done in Figure \ref{fig:example}. In any case the resulting quotient diagram of $D'$ is alternating.
\end{proof}
We can now prove Theorem \ref{thm:quotient}.
\begin{proof}[Proof of Theorem \ref{thm:quotient}]
We first observe that by \cite{composite} a $p$-period $\tau$ on an alternating (not necessarily prime) link $L$ respects a direct sum decomposition of $L$ into prime links. Specifically, for each prime component $C$, $\tau$ either induces a $p$-period on $C$ or else permutes $p$ isomorphic copies of $C$. Furthermore, the quotient of $L$ is a connect sum of components and quotients of components of $L$. Thus for $p$ prime, we immediately have that the quotient of $L$ is alternating by either Corollary \ref{cor:quotient} or one of \cite[Corollary 1.3]{B} or \cite{CH}. 

For the general case, we induct on the number of prime factors of $p$. If $p = q \cdot r$ with $r$ prime, then $L$ quotiented by the $\Z/q\Z$ action is a non-split $r$-periodic alternating link $L'$ by the inductive assumption. Furthermore $r$ is prime, so the quotient of $L'$ is also alternating.  
\end{proof}
\begin{corollary}\label{cor:free}
A freely 2-periodic alternating prime non-split link must have an even number of components. In particular, there are no freely 2-periodic alternating prime knots. 
\end{corollary}
\begin{proof}
Observe that the transformations from Figure \ref{fig:transform1} and Figure \ref{fig:transform2} cannot create freely periodic diagrams, since they are only local replacements and the original intravergent or transvergent diagram was not freely periodic. Hence for a freely periodic involution there is an alternating freely periodic diagram by Theorem \ref{thm:main}. Only the half twist on 1 or 2 strands can be alternating, but a freely periodic diagram on 1 strand is not prime (recall that the unknot is not prime), and a freely periodic diagram on 2 strands has an even number of components. See Figure \ref{fig:free}.
\end{proof}
In order to prove Theorem \ref{thm:main} we will need the following theorem of Menasco and Thistlethwaite. We use the language of realized diagrams rather than the original language of flat homeomorphisms since it is necessary for us to keep track of crossing balls. Hence the following theorem is slightly weaker than the original. 

\begin{theorem}\cite[a slight weakening of the Main Theorem]{MT}
\label{thm:mt}
For any reduced alternating diagram $D$ for $L$ and realization $\lambda(D)$, any homeomorphism of pairs $(S^3, L) \cong (S^3, L)$ is isotopic through maps of pairs to a composition of flypes.
\end{theorem}
In our situation, take the homeomorphism of pairs to be an involution $\tau$, so that $\tau$ is isotopic to $F_{\circ}$ for $F = (f_n, f_{n-1}, \dots, f_1) : \lambda(D) \to \lambda(D)$. In particular, $F_{\circ}^2$ is isotopic to the identity, and we will label the first iteration of $f_i$ as $f_{(i,1)}$ and the second iteration of $f_i$ as $f_{(i,2)}$. With this setup, we give the following definition and lemma.
\begin{definition}
Given a collection of composable flypes $F = (f_n, f_{n-1}, \dots, f_1)$, consider the equivalence relation on these flypes generated by $f_a \sim f_b$ if $[c(f_a)]_F = [d(f_b)]_F$. Then let the \emph{orbit} of $f_i$ be the set of flypes in the equivalence class of $f_i$. A flype $f_j$ is then in the orbit of $f_i$ if there is a sequence of flypes $f_i = f_{i_1},f_{i_2},\dots,f_{i_k} = f_j$ with $[c(f_{i_r})]_F = [d(f_{i_{r+1}})]_F$ for all $r$, or vice versa.
\end{definition}
The following lemma follows from Theorem \ref{thm:mt}, and is similar to \cite[Lemma 3.3]{B} but for involutions instead of odd prime order group actions. 

\begin{lemma}
\label{lemma:orbit}
If $L$ is an alternating prime non-split non-torus link with an involution $\tau$, then there is a reduced alternating diagram $D$ for $L$ and a reduced collection of composable flypes $F = (f_n, \dots, f_1)$ with $f_{\circ} = \tau: \lambda(D) \to \lambda(D)$ such that the orbit of the flype $f_{(i,1)}$ in $F^2 = (f_{(n,2)},\dots,f_{(1,2)},f_{(n,1)},\dots, f_{(1,1)})$ is $\{f_{(i,1)},f_{(i,2)}\}$. 
\end{lemma}
\begin{proof}
Let $F = (f_n,f_{n-1},\dots,f_1)$ be a reduced collection of composable flypes with $F_{\circ}$ isotopic to $\tau$. Such a collection exists by Theorem \ref{thm:mt}.  

Suppose that in the orbit of $f_{(i,1)}$ in $F^2$ we have a flype $f_j \neq f_{(i,2)}$. Then $f_{(i,1)}$ and $f_j$ are both contained in either $G = (f_{(i-1,2)},\dots,f_{(i,1)})$ or $H = (f_{(i,1)},\dots,f_{(i+1,2)})$, since $f_j \neq f_{(i,2)}$. Suppose without loss of generality that $f_j \in G$, and note that $G_{\circ}$ is isotopic to $\tau$, although on a different diagram than $F$. Then cx$(F) = $ cx$(G)$, so $G$ is also reduced. Now apply Proposition \ref{prop:permute} to $G$ to obtain a new reduced collection $G'$ isotopic to $\tau$ with $f'_{(i,1)}$ and $f'_j$ adjacent and $c(f'_{(i,1)}) = d(f'_j)$. This is a contradiction since Lemma \ref{lemma:combine} then implies that $G'$ is not reduced. Thus there is no $f_j \neq f_{(i,2)}$ in the orbit of $f_{(i,1)}$. 

We will also show that there must be at least one other element in the orbit of $f_{(i,1)}$ so that the orbit is exactly $\{f_{(i,1)},f_{(i,2)}\}$. First, observe that by \cite{nontorus} an alternating prime non-split non-torus link is hyperbolic. Then by Mostow rigidity, any finite order map on $(S^3 - L)$ which is isotopic to the identity is the identity. In particular, $F^2_{\circ} = $ id. The identity map preserves crossing balls, however, so there must be another flype $f_j$ in $F^2$ with $[d(f_j)] = [c(f_{(i,1)})]$. 

\end{proof}

Note that while Lemma \ref{lemma:orbit} guarantees that $f_n \circ \dots \circ f_1$ is reduced, $(f_n \circ \dots \circ f_1)^2$ need not be. 

We are now ready to prove the main theorem. 

\begin{proof}[Proof of Theorem \ref{thm:main}]
Let $F = (f_n, \dots, f_1): \lambda(D) \to \lambda(D)$ be a collection of composable flypes satisfying Lemma \ref{lemma:orbit}. We will proceed by induction on $n$. The base case is $n=0$, in which case $F_{\circ}$ is intravergent, transvergent, or freely periodic by Lemma \ref{lemma:floop}. For the inductive step we will show that an innermost flype $f_i$ is of type A or type B. Then a modification of $(L,\tau)$ removing the tangle $T$ (the reverse of one of the moves shown in Figure \ref{fig:transform1} and Figure \ref{fig:transform2}) will produce a new link diagram $D'$ (for a new link $L'$) with a new involution $\tau'$. Because $f_i$ is innermost, the $[c(f_j)]$ and $[d(f_j)]$ are not removed by the modification if $i \neq j$. In particular, $f_j$ induces a flype $\overline{f_j}$ on the modified link $L'$ so that we have a new collection of composable flypes $\overline{F} = (\overline{f_n}, \dots, \widehat{f_i}, \dots, \overline{f_1}) \cong \tau'$ with $n-1$ flypes which is also reduced. Here $\widehat{f_i}$ refers to removing $f_i$ from the collection. 

We now turn to showing that an innermost flype is of type A or type B. Consider an innermost flype in $F$, which exists by Lemma \ref{lemma:innermost}. By a cyclic permutation of the $f_i$ we may assume without loss of generality that $f_1$ is innermost. Let $g = f_n \circ \dots \circ f_2$ so that $F_{\circ} = g \circ f_1$. Then consider the realized diagram for the tangle $T$ in the domain of $f_1$. Because the orbit of $f_1$ in $F_{\circ}^2$ is $\{f_{1,1},f_{1,2}\}$, we see that $g$ fixes $T$ set-wise; since $f_1$ is innermost, $g$ further fixes $T$ diagrammatically. That is, $g$ restricts to an isomorphism of realized diagrams on $T$. By Lemma \ref{lemma:floop} we then have that $g$ acts on $T$ as a transvergent or intravergent involution. Thus $f_1$ is either type A or type B as desired.
\end{proof}

\bibliography{bibliographyA}

\begin{thebibliography}{HTW98}

\bibitem[Boy18]{B2}
Keegan Boyle.
\newblock Rank inequalities on knot {F}loer homology of periodic knots.
\newblock 2018.
\newblock Available at \url{https://arxiv.org/abs/1810.01526}.

\bibitem[Boy19]{B}
Keegan Boyle.
\newblock Odd order group actions on alternating knots.
\newblock 2019.
\newblock Available at \url{https://arxiv.org/abs/1906.04308}.

\bibitem[CH19]{CH}
Antonio~F. Costa and C.~V.~Q. Hongler.
\newblock Periodic projections of alternating knots, 2019.
\newblock Available at \url{https://arxiv.org/abs/1905.13718}.

\bibitem[Chb97]{C}
Nafaa Chbili.
\newblock On the invariants of lens knots.
\newblock In {\em K{NOTS} '96 ({T}okyo)}, pages 365--375. World Sci. Publ.,
  River Edge, NJ, 1997.

\bibitem[Hen15]{H}
Kristen Hendricks.
\newblock Localization of the link {F}loer homology of doubly-periodic knots.
\newblock {\em J. Symplectic Geom.}, 13(3):545--608, 2015.

\bibitem[HLS16]{HLS}
Kristen Hendricks, Robert Lipshitz, and Sucharit Sarkar.
\newblock A flexible construction of equivariant {F}loer homology and
  applications.
\newblock {\em J. Topol.}, 9(4):1153--1236, 2016.

\bibitem[HTW98]{HTW}
Jim Hoste, Morwen Thistlethwaite, and Jeff Weeks.
\newblock The first 1,701,936 knots.
\newblock {\em Math. Intelligencer}, 20(4):33--48, 1998.

\bibitem[JN16]{JN}
Stanislav Jabuka and Swatee Naik.
\newblock Periodic knots and {H}eegaard {F}loer correction terms.
\newblock {\em J. Eur. Math. Soc. (JEMS)}, 18(8):1651--1674, 2016.

\bibitem[Kaw90]{Kaw}
Akio Kawauchi.
\newblock {\em A Survey of Knot Theory}.
\newblock Springer-Verlag, 1990.

\bibitem[LW19]{LW}
Andrew Lobb and Liam Watson.
\newblock A refinement of {K}hovanov homology.
\newblock 2019.
\newblock Available at \url{https://arxiv.org/abs/1908.00082}.

\bibitem[MB84]{MB}
John~W. Morgan and Hyman Bass.
\newblock {\em The Smith conjecture}.
\newblock Academic Press, 1984.

\bibitem[Men84]{nontorus}
W.~Menasco.
\newblock Closed incompressible surfaces in alternating knot and link
  complements.
\newblock {\em Topology}, 23(1):37--44, 1984.

\bibitem[MT93]{MT}
William Menasco and Morwen Thistlethwaite.
\newblock The classification of alternating links.
\newblock {\em Ann. of Math. (2)}, 138(1):113--171, 1993.

\bibitem[PS20]{PS}
Luisa Paoluzzi and Makoto Sakuma.
\newblock Prime amphicheiral knots with free period 2.
\newblock {\em Proc. Edinb. Math. Soc. (2)}, 63(1):105--138, 2020.

\bibitem[PY03]{PY}
J\'{o}zef~H. Przytycki and Akira Yasukhara.
\newblock Symmetry of links and classification of lens spaces.
\newblock {\em Geom. Dedicata}, 98:57--61, 2003.

\bibitem[Sak81]{composite}
Makoto Sakuma.
\newblock Periods of composite links.
\newblock {\em Math. Sem. Notes Kobe Univ.}, 9(2):445--452, 1981.

\bibitem[Sak87]{S}
Makoto Sakuma.
\newblock Non-free-periodicity of amphicheiral hyperbolic knots.
\newblock In {\em Homotopy theory and related topics ({K}yoto, 1984)}, volume~9
  of {\em Adv. Stud. Pure Math.}, pages 189--194. North-Holland, Amsterdam,
  1987.

\bibitem[Smi39]{smith}
P.~A. Smith.
\newblock Transformations of finite period. {II}.
\newblock {\em Ann. of Math. (2)}, 40:690--711, 1939.

\end{thebibliography}
\bibliographystyle{amsplain}

\end{document}